\documentclass{amsart}
\usepackage{amsmath,amssymb,amsthm}
\usepackage{graphicx}
\usepackage{braket}

\newtheorem{theorem}{Theorem}[section]
\newtheorem{proposition}[theorem]{Proposition}
\newtheorem{corollary}[theorem]{Corollary}

\theoremstyle{definition}
\newtheorem{definition}[theorem]{Definition}
\newtheorem{remark}[theorem]{Remark}
\newtheorem{example}[theorem]{Example}

\numberwithin{equation}{section}%

\newcommand{\flow}{\mathrm{Flow}}
\newcommand{\col}{\operatorname{Col}}
\newcommand{\type}{\operatorname{type}}
\newcommand{\pr}{\mathrm{pr}}
\newcommand{\Z}{\mathbb{Z}}

\makeatletter
\@namedef{subjclassname@2020}{\textup{2020} Mathematics Subject Classification}
\makeatother
\begin{document}
\title[An MGR and oriented spatial surfaces]{A multiple group rack and oriented spatial surfaces}
\author[A.~Ishii]{Atsushi Ishii}
\address{Institute of Mathematics, University of Tsukuba,
1-1-1 Tennodai, Tsukuba, Ibaraki 305-8571, Japan}
\email{aishii@math.tsukuba.ac.jp}
\author[S.~Matsuzaki]{Shosaku Matsuzaki}
\address[S.~Matsuzaki]{Liberal Arts Education Center, Ashikaga University, 268-1 Ohmae-cho, Ashikaga-shi, Tochigi 326-8558, Japan}
\email{matsuzaki.shosaku@g.ashikaga.ac.jp}
\author[T.~Murao]{Tomo Murao}
\address[T.~Murao]{Science and Technology Unit, Natural Sciences Cluster, Research and Education Faculty, Kochi University, 2-5-1 Akebono-cho, Kochi-shi, Kochi 780-8072, Japan}
\email{tmurao@kochi-u.ac.jp}
\maketitle

\begin{abstract}
A spatial surface is a compact surface embedded in the $3$-sphere.
In this paper, we provide several typical examples of spatial surfaces and construct a coloring invariant to distinguish them.
The coloring is defined by using a multiple group rack, which is a rack version of a multiple conjugation quandle.
\end{abstract}


\section{Introduction}
A spatial surface is a compact surface embedded in the $3$-sphere.
If it has a boundary, the spatial surface is a Seifert surface of the link given by the boundary.
Recently, the second author~\cite{Matsuzaki19} established the Reidemeister moves for spatial surfaces with a boundary.
This enables us to study Seifert surfaces through their combinatorial structures.
On the other hand, a spatial surface without a boundary are closely related to handlebody-knots~\cite{Ishii08}, which is a handlebody embedded in the $3$-sphere.
A connected spatial surface divides the $3$-sphere into two $3$-manifolds, if the surface has no boundary.
When one of the $3$-manifolds is homeomorphic to a handlebody, the spatial surface represents a handlebody-knot.
In this sense, a spatial surface is a generalization of a handlebody-knot.

In this paper, we provide several typical examples of oriented spatial surfaces and construct a coloring invariant to distinguish them.
A coloring invariant for knots is defined by using a quandle~\cite{Joyce82,Matveev82}, whose axioms are derived from the Reidemeister moves for knots.
In a similar manner, a multiple conjugation quandle~\cite{Ishii15-1} is used to define a coloring invariant for handlebody-knots.
In this paper, we introduce a multiple group rack to define a coloring invariant for oriented spatial surfaces.
A multiple group rack is a rack version of a multiple conjugation quandle, that is, the axioms of a multiple group rack is obtained by removing an axiom from those of a multiple conjugation quandle.

In general, it is not easy to find a multiple group rack.
In this paper, we introduce a $G$-family of racks and show that it induces a multiple group rack, where we remark that a $G$-family of racks can be obtained from a rack with its type.
With the multiple group rack obtained from a $G$-family of racks, we can enhance the coloring invariant for oriented spatial surfaces.
Calculation examples provided in this paper are given by using this enhanced coloring invariant.
Some examples do not need coloring invariants to distinguish them.
We demonstrate that the boundary and regular neighborhood work well for some spatial surfaces.

As mentioned above, a rack and its type play important roles to evaluate coloring invariants of oriented spatial surfaces.
In this paper, we also provide a method to construct racks from a rack (and quandle).
More precisely, for a rack $X$, we obtain a new rack $X^n$ that is not a direct product of $X$'s.
We give an example of spatial surfaces that can be distinguished by using $R_3^3$, not $R_3$, where $R_3$ is the dihedral quandle of order $3$.

This paper is organized as follows.
In the first half of Section~2, we recall spatial surfaces and the Reidemeister moves for them.
In the second half of Section~2, we demonstrate that the boundary and regular neighborhood of a spatial surface can be used to distinguish some spatial surfaces.
In Section~3, we recall the definition of a rack and introduce a multiple group rack.
In Section~4, we define a coloring invariant with a multiple group rack.
In Section~5, we distinguish some spatial surfaces by using the coloring invariants.


\section{Oriented spatial surfaces and their regular neighborhoods and\\boundaries}

A \textit{spatial surface} is a compact surface embedded in the 3-sphere $S^3=\mathbb R^3 \cup \{ \infty \}$.
Oriented spatial surfaces $F$ and $F'$ are \textit{equivalent}, denoted by $F \cong F'$, 
if there is an orientation-preserving self-homeomorphism $h:S^3 \to S^3$ such that $h(F)=F'$, where we note that the orientation of $h(F)$ also coincides with that of $F'$.
The oriented spatial surface obtained by reversing the orientation of an oriented spatial surface $F$ is the \textit{reverse} of $F$.
An oriented spatial surface $F$ is \textit{reversible} if $F$ and its reverse are equivalent.
Unless otherwise noted, we suppose the following conditions:
(i) each component of a spatial surface has non-empty boundary, and
(ii) a spatial surface has no disk components.
In figures of this paper, the front side and the back side of an oriented spatial surface are colored by light gray and dark gray, respectively (see Fig.~\ref{FIG:color_gray}).
		\begin{figure}[htbp]
		\begin{center}
		\includegraphics{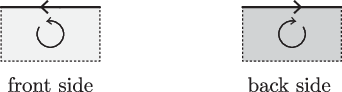}
		\caption{The front side and the back side.}
		\label{FIG:color_gray}
		\end{center}
		\end{figure}

A {\it spatial trivalent graph} is a trivalent graph embedded in $S^3$.
Let $D$ be a diagram of a spatial trivalent graph.
We obtain a spatial surface $F$ from $D$ by taking a regular neighborhood of $D$ in $\mathbb R^2$ and perturbing it around all crossings of $D$, according to its over/under information.
Then we give $F$ an orientation so that the front side of $F$ faces into the positive direction of the $z$-axis of $\mathbb R^3$  (see Fig.~\ref{FIG:construction_of_spatial_surface}).
Then we say that $D$ \textit{represents} the oriented spatial surface $F$ and call $D$ a \textit{diagram} of $F$.
We remark that any oriented spatial surface is equivalent to an oriented spatial surface obtained by this process.
		\begin{figure}[htbp]
		\begin{center}
		\includegraphics{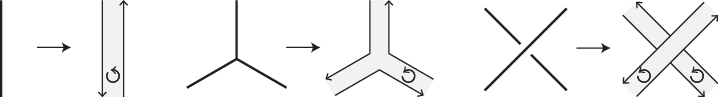}
		\caption{The process for obtaining an oriented spatial surface.}
		\label{FIG:construction_of_spatial_surface}
		\end{center}
		\end{figure}

\begin{theorem}[\cite{Matsuzaki19}]\label{theorem_oriented_surface}
Two oriented spatial surfaces are equivalent if and only if their diagrams are related by the Reidemeister moves on $S^2$, which are depicted in Fig.~\ref{FIG:spatial_surface_Reidemeister_move}.
\end{theorem}
		\begin{figure}[htbp]
		\begin{center}
		\includegraphics{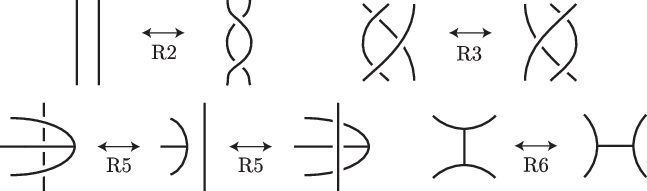}
		\caption{The Reidemeister moves for oriented spatial surfaces.}
		\label{FIG:spatial_surface_Reidemeister_move}
		\end{center}
		\end{figure}

Next, we introduce elementary methods to distinguish oriented spatial surfaces and show some examples.

A \textit{handlebody-link} is the disjoint union of handlebodies embedded in $S^3$.
A \textit{handlebody-knot} is a handlebody-link with one component.
Two handlebody-links are \textit{equivalent} if there is an orientation-preserving self-homeomorphism of $S^3$ sending one to the other.
Let $F$ be an oriented spatial surface.
The regular neighborhood $N(F)$ of $F$ in $S^3$ is a handlebody-link.
We denote by $\partial F$ the boundary of $F$, where we assume that $\partial F$ is oriented so that the orientation is coherent with that of $F$.

\begin{remark}\label{rem:elementary method}
If oriented spatial surfaces $F_1$ and $F_2$ are equivalent, then
\begin{enumerate}
\item
$\partial F_1$ and $\partial F_2$ are equivalent as oriented links, and
\item
$N(F_1)$ and $N(F_2)$ are equivalent as handlebody-links.
\end{enumerate}
\end{remark}

In Examples~\ref{ex:theta} and~\ref{ex:handlebody}, we distinguish oriented spatial surfaces from this observation.

\begin{example}\label{ex:theta}
Let $F(i,j,k)$ be an oriented spatial surface as illustrated in Fig.~\ref{FIG:example_theta} ($i,j,k \in \mathbb Z$), where the integers $i, j$ and $k$ respectively indicate $i,j$ and $k$ full-twists.
We note that a negative integer indicates the reverse twists.
Then, the following conditions are equivalent:
\begin{enumerate}
\item
the multiset $\{i,j,k\}$ coincides with the multiset $\{i',j',k'\}$,
\item
$F(i,j,k)$ and $F(i',j',k')$ are equivalent, and
\item
$\partial F(i,j,k)$ and $\partial F(i',j',k')$ are equivalent as oriented links.
\end{enumerate}
If the condition (1) holds, we can check that $F(i,j,k)$ and $F(i',j',k')$ are equivalent by rotating the sphere and by turning the sphere to be upside down.
From Remark~\ref{rem:elementary method} (1), the condition (2) implies (3).

For $\partial F=K_1 \cup K_2 \cup K_3$, the multiset
\[ \operatorname{lk}(\partial F):=\{\operatorname{lk}(K_1, K_2), \operatorname{lk}(K_2, K_3), \operatorname{lk}(K_3, K_1)\} \]
 of  the linking numbers of $K_1 \cup K_2$, $K_2 \cup K_3$ and $K_3 \cup K_1$ is an invariant of $\partial F$.
Since $\operatorname{lk}(\partial F(i,j,k))=\{i,j,k\}$, the condition (1) follows from (3).
\end{example}
	\begin{figure}[htbp]
	\begin{center}\includegraphics{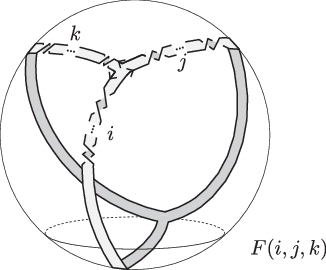}
	\caption{An oriented spatial surface $F(i,j,k)$ whose boundary has three components.}\label{FIG:example_theta}
	\end{center}
	\end{figure}

\begin{example}\label{ex:handlebody}
Let $F_1$ and $F_2$ be the spatial surfaces illustrated in Fig.~\ref{fig:example_handlebody-knot}.
The oriented links $\partial F_1$ and $\partial F_2$ are equivalent:
both of them are trivial knots.
Then we consider their regular neighborhoods.
It is easy to see that $N(F_1)$ is the genus 2 trivial handlebody-knot, where a trivial handlebody-knot is a handlebody standardly embedded in $S^3$.
On the other hand, $N(F_2)$ is not a trivial handlebody-knot~\cite{IshiiKishimotoMoriuchiSuzuki12}.
Therefore $F_1$ and $F_2$ are not equivalent, see Remark~\ref{rem:elementary method} (2).
\end{example}
	\begin{figure}[htbp]
	\begin{center}\includegraphics{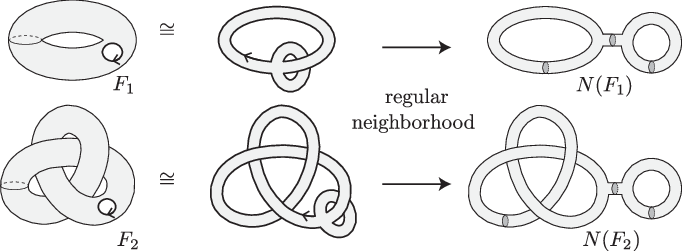}
	\caption{Inequivalent spatial surfaces.}\label{fig:example_handlebody-knot}
	\end{center}
	\end{figure}

\begin{remark}\label{uniqueness}
Let $\widetilde{F_1}$ and $\widetilde{F_2}$ be oriented spatial closed surfaces illustrated in Fig.~\ref{fig:disk close}, from which $F_1$ and $F_2$ in Fig.~\ref{fig:example_handlebody-knot} are obtained by removing a disk, respectively.
Then $\widetilde{F_1}$ and $\widetilde{F_2}$ are not equivalent, since $\widetilde{F_1}$ splits $S^3$ into two solid tori, but $\widetilde{F_2}$ does not.
This also implies that $F_1$ and $F_2$ are not equivalent by~\cite{Matsuzaki19}.

\end{remark}
	\begin{figure}[htpb]
	\begin{center}
	\includegraphics{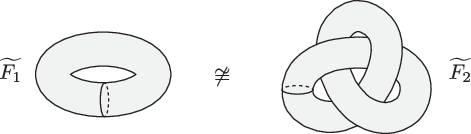}
	\caption{Inequivalent spatial closed surfaces.}\label{fig:disk close}
	\end{center}
	\end{figure}

The spatial surfaces $F_1$ and $F_2$ depicted in Fig.~\ref{fig:not invertible2} can not be distinguished by using the elementary methods introduced in Remarks~\ref{rem:elementary method} and~\ref{uniqueness}.
In Section~5, we show that they are not equivalent by using coloring invariants which we will introduce in Section~4.


\section{A multiple group rack}

In this section, we introduce a notion of multiple group racks and give examples of them.

A \textit{rack}~\cite{FennRourke92} is a non-empty set $X$ with a binary operation $*:X\times X\to X$ satisfying the following axioms:
\begin{itemize}
\item
For any $a\in X$, the map $S_a:X\to X$ defined by $S_a(x)=x*a$ is a bijection.
\item
For any $x,y,z\in X$, $(x*y)*z=(x*z)*(y*z)$.
\end{itemize}
A rack $X$ is called a {\it quandle}~\cite{Joyce82,Matveev82} if it satisfies the following axiom:
\begin{itemize}
\item
For any $x \in X$, $x*x=x$.
\end{itemize}
Let $X=(X,*)$ be a rack.
We denote $S_a^k(x)$ by $x*^k a$ for any $a, x \in X$ and $k \in \Z$, where we note that $S_a^0=\mathrm{id}_X$.
For any $x, x_1, \ldots, x_m \in X$ and $k_1, \ldots, k_n \in \Z$, we often abbreviate parentheses and write $S_{x_n}^{k_n}\circ\cdots\circ S_{x_2}^{k_2}\circ S_{x_1}^{k_1}(x)$ by $x*^{k_1}x_1*^{k_2}x_2\cdots*^{k_n}x_n$.

We give some examples of racks and quandles.
Let $G$ be a group.
We define a binary operation $*:G \times G \to G$ by $a* b:=b^{-1}ab$ for any $a,b \in G$.
Then $(G,*)$ is a quandle, called the \textit{conjugation quandle} of $G$ and denoted by $\operatorname{Conj}G$.
For a positive integer $n$, we denote by $\mathbb{Z}_n$ the cyclic group $\mathbb{Z}/n\mathbb{Z}$ of order $n$.
We define a binary operation $*:\Z_n \times \Z_n \to \Z_n$ by $a*b:=2b-a$ for any $a,b \in \Z_n$.
Then $(\mathbb{Z}_n,*)$ is a quandle, called the \textit{dihedral quandle} of order $n$ and denoted by $R_n$.
We define a binary operation $*:\Z_n \times \Z_n \to \Z_n$ by $a*b:=a+1$ for any $a,b \in \Z_n$.
Then $(\mathbb{Z}_n,*)$ is a rack, called the \textit{cyclic rack} of order $n$ and denoted by $C_n$.
Let $R$ be a ring and $M$ a left $R[t^{\pm 1},s]/(s(t+s-1))$-module.
We define a binary operation $*:M \times M \to M$ by $x*y:=tx+sy$ for any $x,y \in M$.
Then $(M,*)$ is a rack, called the \textit{$(t,s)$-rack}.
When $s=1-t$, the $(t,s)$-rack is called the \textit{Alexander quandle}.

\begin{proposition}\label{prop:quandle rack}
Let $X$ be a rack.
Fix $e_1,\ldots,e_m\in\mathbb{Z}$ and $i_1,\ldots,i_m\in\{1,\ldots,n\}$.
Then $X^n$ is a rack with the binary operation defined by
\[
(x_1,\ldots,x_n)*(y_1,\ldots,y_n)
=(x_1*^{e_1}y_{i_1}*^{e_2}\cdots*^{e_m}y_{i_m},\ldots,x_n*^{e_1}y_{i_1}*^{e_2}\cdots*^{e_m}y_{i_m}).
\]
\end{proposition}

\begin{proof}
Since $S_a:X\to X$ is bijective, so is the map $S_{(a_1,\ldots,a_n)}:X^n\to X^n$ defined by $S_{(a_1,\ldots,a_n)}(x_1,\ldots,x_n)=(x_1,\ldots,x_n)*(a_1,\ldots,a_n)$.
We have
\begin{align*}
&(x_0*^{s_1}x_1*^{s_2}\cdots*^{s_p}x_p)*^{t_0}(y_0*^{t_1}y_1*^{t_2}\cdots*^{t_q}y_q)\\
&=x_0*^{s_1}x_1*^{s_2}\cdots*^{s_p}x_p*^{-t_q}y_q*^{-t_{q-1}}y_{q-1}*^{-t_{q-2}}\cdots*^{-t_1}y_1*^{t_0}y_0*^{t_1}\cdots*^{t_q}y_q,
\end{align*}
which implies
\[
(x_0*^{s_1}z_1*^{s_2}\cdots*^{s_p}z_p)
*^{t_0}(y_0*^{s_1}z_1*^{s_2}*\cdots*^{s_p}z_p)=x_0*^{t_0}y_0*^{s_1}z_1*^{s_2}\cdots*^{s_p}z_p
\]
for $x_0,\ldots,x_p,y_0,\ldots,y_q,z_1,\ldots,z_p\in X$ and $s_1,\ldots,s_p,t_0,\ldots,t_q\in\mathbb{Z}$.
This equation implies the second axiom of a rack.
\end{proof}

We note that, in general, $X^n$ in Proposition~\ref{prop:quandle rack} are not quandles even if $X$ is a quandle.

We define the \textit{type} of a rack $X$, denoted by $\type X$, by the minimal number of $n \in \Z_{>0}$ satisfying $a*^nb=a$ for any $a,b \in X$. 
We set $\type X:= \infty$ if we do not have such a positive integer $n$.
Any finite rack is of finite type since the set $\{S_a^n \mid n\in \Z \}$ is finite for any $a \in X$ when $X$ is finite.
For example, we can easily check that $\type C_n=n$ and $\type R_n=2$ for any positive integer $n$.

\begin{proposition}\label{prop:rack type}
Let $R_q$ be the dihedral quandle of order $q\geq3$.
We give $R_q^n$ the binary operation defined by
\[
(x_1,\ldots,x_n)*(y_1,\ldots,y_n)
=(x_1*y_1*\cdots*y_n,\ldots,x_n*y_1*\cdots*y_n),
\]
which is a rack as we saw in Proposition~\ref{prop:quandle rack}.
Then we have
\begin{align*}
\type R_q^n=
\begin{cases}
2 &\text{if $n$ is odd},\\
q &\text{if $n$ is even, and $q$ is odd},\\
q/2 &\text{if $n$ and $q$ are even}.
\end{cases}
\end{align*}
\end{proposition}

\begin{proof}
This follows from 
\begin{align*}
&(a_1,\ldots,a_n)*(b_1,\ldots,b_n) \\
&=(2(b_n-b_{n-1}+\cdots+b_1)-a_1,\ldots,2(b_n-b_{n-1}+\cdots+b_1)-a_n)
\end{align*}
for odd integer $n$, and
\begin{align*}
&(a_1,\ldots,a_n)*^k(b_1,\ldots,b_n) \\
&=(2k(b_n-b_{n-1}+\cdots-b_1)+a_1,\ldots,2k(b_n-b_{n-1}+\cdots-b_1)+a_n)
\end{align*}
for even integer $n$.
\end{proof}

\begin{definition}
A \textit{multiple group rack} $X=\bigsqcup_{\lambda\in\Lambda}G_\lambda$ is a disjoint union of groups $G_\lambda$ ($\lambda\in\Lambda$) with a binary operation $*:X\times X\to X$
satisfying the following axioms:
\begin{itemize}
\item For any $x\in X$ and $a,b\in G_\lambda$,
\[x*(ab)=(x*a)*b \text{ and } x*e_\lambda=x,\]
where $e_\lambda$ is the identity of $G_\lambda$.
\item For any $x,y,z\in X$,
\[
(x*y)*z=(x*z)*(y*z).
\]
\item For any $x\in X$ and $a,b\in G_\lambda$,
\[
(ab)*x=(a*x)(b*x),
\]
where $a*x,b*x\in G_\mu$ for some $\mu\in\Lambda$.
\end{itemize}
\end{definition}
We remark that a multiple group rack is a rack.

\begin{remark}
A multiple group rack $X=\bigsqcup_{\lambda\in\Lambda}G_\lambda$ is called a \textit{multiple conjugation quandle} if it satisfies the following axiom:
\begin{itemize}
\item
For any $a,b \in G_\lambda$, $a*b=b^{-1}ab$.
\end{itemize}
A multiple conjugation quandle is an algebra introduced in~\cite{Ishii15-1} to define coloring invariants for handlebody-knots.
A multiple conjugation quandle is a multiple group rack clearly.
\end{remark}

Next, we define a $G$-family of racks.
This is an algebra yielding a multiple group rack.

\begin{definition}
Let $G$ be a group with identity element $e$.
A \textit{$G$-family of racks} $X$ is a set with a family of binary operations $*^g:X \times X \to X$ ($g\in G$) satisfying the following axioms:
\begin{itemize}
\item
For any $x,y\in X$ and $g,h\in G$, \[x*^{gh}y=(x*^gy)*^hy\text{ and }x*^ey=x.\]
\item
For any $x,y,z\in X$ and $g,h\in G$, \[(x*^gy)*^hz=(x*^hz)*^{h^{-1}gh}(y*^hz).\]
\end{itemize}
\end{definition}

A $G$-family of racks $X$ is called a {\it $G$-family of quandles}~\cite{IshiiIwakiriJangOshiro13} if it satisfies the following axiom:
\begin{itemize}
\item
For any $x \in X$ and $g \in G$, $x*^gx=x$.
\end{itemize}
The following propositions state that a rack gives some $G$-families of racks, and a $G$-family of racks yields a multiple group rack.
They can be verified by direct calculation (see~\cite{IshiiIwakiriJangOshiro13} for more details).

\begin{proposition}\label{prop:G-family of racks}
Let $(X,*)$ be a rack.
Then $(X,\{*^i\}_{i \in \Z})$ is a $\Z$-family of racks.
Furthermore, if $\type X$ is finite, $(X,\{*^i\}_{i \in \Z_{\type X}})$ is a $\Z_{\type X}$-family of racks.
\end{proposition}

\begin{proposition}\label{prop:MGR from G-family}
Let $G$ be a group and $(X,\{*^g\}_{g\in G})$ a $G$-family of racks.
Then $\bigsqcup_{x\in X} \left( \{x\}\times G \right)$ is a multiple group rack with
\begin{align*}
&(x,g)*(y,h)=(x*^hy, h^{-1}gh), &
&(x,g)(x,h)=(x,gh)
\end{align*}
for any $x,y\in X$ and $g,h\in G$.
\end{proposition}

The multiple group rack $\bigsqcup_{x\in X} \left( \{x\}\times G \right)$ in Proposition~\ref{prop:MGR from G-family} is called the {\it associated multiple group rack} of a $G$-family of racks $(X,\{*^g\}_{g\in G})$.

\section{A coloring invariant for oriented spatial surfaces}

In this section, we introduce a coloring for an oriented spatial surface by a multiple group rack.

Let $D$ be a diagram of an oriented spatial surface.
A \textit{Y-orientation} of $D$ is a collection of orientations of all edges of $D$ without sources and sinks with respect to the orientation as shown in Fig.~\ref{fig:Y-orientation}, where an edge of $D$ is a piece of a curve each of whose endpoints is a vertex.
In this paper, a circle component of $D$ is also regarded as an edge of $D$.
Every diagram has a Y-orientation.
We note that a Y-orientation of $D$ is irrelevant to the orientations of the spatial surface represented by $D$.
\textit{Y-oriented Reidemeister moves} are Reidemeister moves between two diagrams with Y-orientations which are identical except in the disk where the move is applied.
All Y-oriented R6 moves are listed in Fig.~\ref{fig:Y-oriented R6 moves}.
We have the following proposition immediately by~\cite{Ishii15-2}.

\begin{figure}[htbp]
\mbox{}\hfill
\begin{picture}(60,60)
\put(0,15){\line(1,0){60}}
\put(0,15){\line(2,1){30}}
\put(0,15){\line(2,3){30}}
\put(60,15){\line(-2,1){30}}
\put(60,15){\line(-2,3){30}}
\put(30,30){\line(0,1){30}}
\thicklines
\put(34,15){\vector(1,0){0}}
\put(19,24.5){\vector(2,1){0}}
\put(13,34.5){\vector(-2,-3){0}}
\put(41,24.5){\vector(-2,1){0}}
\put(43,40.5){\vector(-2,3){0}}
\put(30,48){\vector(0,1){0}}
\put(30,4){\makebox(0,0){Y-orientation}}
\end{picture}
\hfill
\begin{picture}(60,60)
\put(0,15){\line(1,0){60}}
\put(0,15){\line(2,1){30}}
\put(0,15){\line(2,3){30}}
\put(60,15){\line(-2,1){30}}
\put(60,15){\line(-2,3){30}}
\put(30,30){\line(0,1){30}}
\thicklines
\put(34,15){\vector(1,0){0}}
\put(19,24.5){\vector(2,1){0}}
\put(18,42){\vector(2,3){0}}
\put(41,24.5){\vector(-2,1){0}}
\put(43,40.5){\vector(-2,3){0}}
\put(30,48){\vector(0,1){0}}
\put(30,4){\makebox(0,0){non-Y-orientation}}
\end{picture}
\hfill\mbox{}
\caption{A Y-orientation and a non-Y-orientation.}
\label{fig:Y-orientation}
\end{figure}
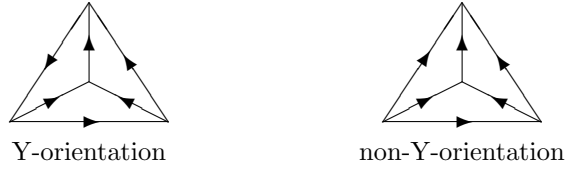

\begin{figure}[htbp]
\begin{center}
\includegraphics{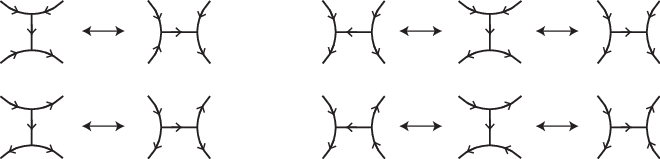}
\caption{All Y-oriented R6 moves.}
\label{fig:Y-oriented R6 moves}
\end{center}
\end{figure}

\begin{proposition}\label{prop:Y-oriented Reidemeister}
Any two Y-orientations of a diagram of an oriented spatial surface can be transformed into each other by using Y-orientated Reidemeister moves finitely and reversing the orientations of some circle components.
\end{proposition}

Let $X=\bigsqcup_{\lambda\in\Lambda}G_\lambda$ be a multiple group rack and let $D$ be a Y-oriented diagram of an oriented spatial surface.
We denote by $\mathcal{A}(D)$ the set of arcs of $D$, where an arc is a piece of a curve each of whose endpoints is an undercrossing or a vertex.
In this paper, an orientation of an arc is also represented by the normal orientation, which is obtained by rotating the usual orientation counterclockwise by $\pi/2$ on the diagram.
An \textit{$X$-coloring} of $D$ is a map $C : \mathcal{A}(D) \to X$ satisfying the conditions depicted in Fig.~\ref{fig:coloring} at each crossing and vertex of $D$.
We denote by $\operatorname{Col}_X(D)$ the set of all $X$-colorings of $D$.
Then we have the following propositions.

\begin{figure}[htb]
\mbox{}\hfill
\begin{picture}(70,80)
 \put(10,50){\line(1,0){17}}
 \put(50,50){\line(-1,0){17}}
 \put(30,70){\vector(0,-1){40}}
 \put(31,65){\makebox(0,0){$\rightarrow$}}
 \put(30,73){\makebox(0,0)[b]{$y$}}
 \put(7,50){\makebox(0,0)[r]{$x$}}
 \put(53,50){\makebox(0,0)[l]{$x*y$}}
\end{picture}
\hfill
\begin{picture}(60,80)
 \put(10,70){\vector(1,-1){20}}
 \put(50,70){\vector(-1,-1){20}}
 \put(30,50){\vector(0,-1){20}}
 \put(15,65){\makebox(0,0){$\nearrow$}}
 \put(45,65){\makebox(0,0){$\searrow$}}
 \put(31,35){\makebox(0,0){$\rightarrow$}}
 \put(8,72){\makebox(0,0)[br]{$a$}}
 \put(52,72){\makebox(0,0)[bl]{$b$}}
 \put(30,27){\makebox(0,0)[t]{$ab$}}
\end{picture}
\hfill
\begin{picture}(60,80)
 \put(30,50){\vector(1,-1){20}}
 \put(30,50){\vector(-1,-1){20}}
 \put(30,70){\vector(0,-1){20}}
 \put(15,35){\makebox(0,0){$\searrow$}}
 \put(45,35){\makebox(0,0){$\nearrow$}}
 \put(31,65){\makebox(0,0){$\rightarrow$}}
 \put(8,28){\makebox(0,0)[tr]{$a$}}
 \put(52,28){\makebox(0,0)[tl]{$b$}}
 \put(30,73){\makebox(0,0)[b]{$ab$}}
 \put(30,5){\makebox(0,0){$x,y \in X,~a,b\in G_\lambda$}}
\end{picture}
\hfill\mbox{}
\caption{The coloring condition.}
\label{fig:coloring}
\end{figure}
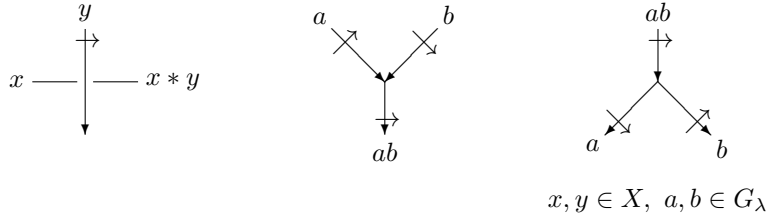

\begin{proposition}\label{prop:Y-orientation}
Let $X=\bigsqcup_{\lambda \in \Lambda}G_\lambda$ be a multiple group rack and let $D$ be a Y-oriented diagram of an oriented spatial surface.
Let $D'$ be a diagram obtained by applying one of Y-oriented Reidemeister moves to $D$ once.
For any $X$-coloring $C$ of $D$, there is a unique $X$-coloring $C'$ of $D'$ which coincides with $C$ except in the disk where the move is applied.
\end{proposition}

\begin{proof}
The assignment of elements of $X$ to the arcs in the disk where the move is applied is uniquely determined by those of the other arcs.
The well-definedness of the resulting $X$-coloring follows from  
\[x*(ab)=(x*a)*b,~ x*e_\lambda=x \]
for Y-oriented R2 moves, and 
\[(x*y)*z=(x*z)*(y*z) \]
for Y-oriented R3 moves, and 
\[x*(ab)=(x*a)*b,~ (ab)*x=(a*x)(b*x) \]
for Y-oriented R5 moves, and the associativity of multiplication for Y-oriented R6 moves.
We illustrate a few cases of Y-oriented Reidemeister moves in Fig.~\ref{fig:proof_coloring}.
\end{proof}

\begin{figure}[htbp]
\begin{center}
\includegraphics[width=120mm]{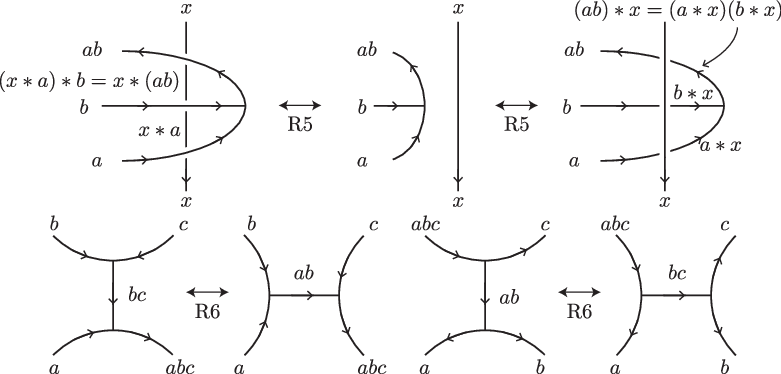}
\caption{$X$-colorings for Y-oriented Reidemeister moves.}
\label{fig:proof_coloring}
\end{center}
\end{figure}

\begin{proposition}\label{prop:S^1-orientation}
Let $X=\bigsqcup_{\lambda \in \Lambda}G_\lambda$ be a multiple group rack 
and let $D$ be a Y-oriented diagram of an oriented spatial surface.
Let $D'$ be a diagram obtained by reversing the orientations of some circle components of $D$.
For an $X$-coloring $C$ of $D$, we define the map $C' : \mathcal{A}(D') \to X$ by
\begin{align*}
C'(\alpha)=
\begin{cases}
C(\alpha) &(\text{the orientation of $\alpha$ in $D'$ coincides with that in $D$}),\\
C(\alpha)^{-1} &(\text{otherwise})
\end{cases}
\end{align*}
for any $\alpha \in \mathcal{A}(D')=\mathcal{A}(D)$.
Then $C'$ is an $X$-coloring of $D'$, and the map from $\col_X(D)$ to $\col_X(D')$ sending $C$ into $C'$ is a bijection.
\end{proposition}

\begin{proof}
The map $C'$ is an $X$-coloring since 
\begin{align*}
x*^{-1}y=x*y^{-1},~ 
x^{-1}*y=(x*y)^{-1}
\end{align*}
for any $x,y \in X$ (see Fig.~\ref{fig:reverse coloring}).
It is easy to see that the map from $\col_X(D)$ to $\col_X(D')$ sending $C$ into $C'$ is a bijection.
\end{proof}

\begin{figure}[htb]
\begin{center}
\includegraphics{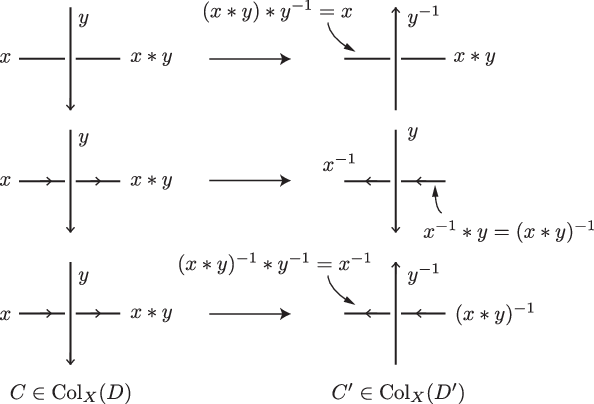}
\caption{The $X$-colorings $C$ and $C'$.}
\label{fig:reverse coloring}
\end{center}	
\end{figure}

Then we have the following theorem 
by Propositions~\ref{prop:Y-oriented Reidemeister},~\ref{prop:Y-orientation} and~\ref{prop:S^1-orientation}.

\begin{theorem}\label{thm:coloring}
Let $X=\bigsqcup_{\lambda\in\Lambda}G_\lambda$ be a multiple group rack.
Let $D$ and $D'$ be Y-oriented diagrams of an oriented spatial surface $F$.
Then there is a one-to-one correspondence between $\col_X(D)$ and $\col_X(D')$.
In particular, the cardinality $\#\col_X(D)$ is an invariant of $F$.
\end{theorem}

In Theorem~\ref{thm:coloring}, if $X$ is an associated multiple group rack of a $G$-family of racks, we can enhance the invariant by using group representations 
of the fundamental group $\pi_1(S^3-F)$ to $G$.

Let $F$ be an oriented spatial surface 
and let $G$ be a group.
A \textit{$G$-flow} of $F$ is a group representation of the fundamental group $\pi_1(S^3-F)$ to $G$, which is a group homomorphism from $\pi_1(S^3-F)$ to $G$.
A $G$-flow of $F$ can also be understood diagrammatically as follows, where we emphasize a $G$-flow of $F$ itself does not depend a diagram of $F$.
Let $D$ be a Y-oriented diagram of $F$.
A \textit{$G$-flow} of $D$ is a map $\varphi : \mathcal{A}(D) \to G$ satisfying the conditions depicted in Fig.~\ref{fig:G-flow} 
at each crossing and vertex of $D$, where $\alpha \in \mathcal{A}(D)$ corresponds to the loop winding around $\alpha$ with the orientation following the right-hand rule.
In this paper, to avoid confusion, we often represent elements of $G$ with underlines.
We call a pair $(D,\varphi)$ of a Y-oriented diagram $D$ and its $G$-flow $\varphi$ a \textit{$G$-flowed diagram} of $F$, and denote by $\flow(D;G)$ the set of all $G$-flows of $D$.
Let $D'$ be another Y-oriented diagram of $F$.
Two $G$-flows $\varphi \in \flow(D;G)$ and $\varphi' \in \flow(D';G)$ 
can be identified when they represent the same $G$-flow of $F$.
In this sense, we can regard $\flow(D;G)=\flow(D';G)$.

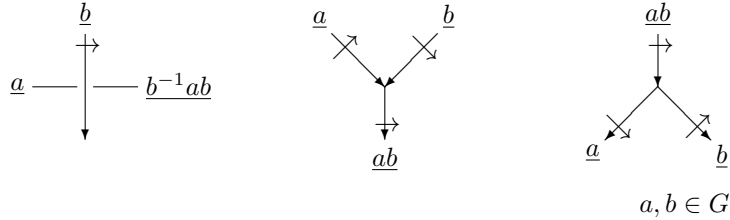
\begin{figure}[htb]
\mbox{}\hfill
\begin{picture}(70,80)
 \put(10,50){\line(1,0){17}}
 \put(50,50){\line(-1,0){17}}
 \put(30,70){\vector(0,-1){40}}
 \put(31,65){\makebox(0,0){$\rightarrow$}}
 \put(30,73){\makebox(0,0)[b]{$\underline{b}$}}
 \put(7,50){\makebox(0,0)[r]{$\underline{a}$}}
 \put(53,50){\makebox(0,0)[l]{$\underline{b^{-1}ab}$}}
\end{picture}
\hfill
\begin{picture}(60,80)
 \put(10,70){\vector(1,-1){20}}
 \put(50,70){\vector(-1,-1){20}}
 \put(30,50){\vector(0,-1){20}}
 \put(15,65){\makebox(0,0){$\nearrow$}}
 \put(45,65){\makebox(0,0){$\searrow$}}
 \put(31,35){\makebox(0,0){$\rightarrow$}}
 \put(8,72){\makebox(0,0)[br]{$\underline{a}$}}
 \put(52,72){\makebox(0,0)[bl]{$\underline{b}$}}
 \put(30,27){\makebox(0,0)[t]{$\underline{ab}$}}
\end{picture}
\hfill
\begin{picture}(60,80)
 \put(30,50){\vector(1,-1){20}}
 \put(30,50){\vector(-1,-1){20}}
 \put(30,70){\vector(0,-1){20}}
 \put(15,35){\makebox(0,0){$\searrow$}}
 \put(45,35){\makebox(0,0){$\nearrow$}}
 \put(31,65){\makebox(0,0){$\rightarrow$}}
 \put(8,28){\makebox(0,0)[tr]{$\underline{a}$}}
 \put(52,28){\makebox(0,0)[tl]{$\underline{b}$}}
 \put(30,73){\makebox(0,0)[b]{$\underline{ab}$}}
 \put(40,5){\makebox(0,0){$a,b\in G$}}
\end{picture}
\hfill\mbox{}
\caption{A $G$-flow.}
\label{fig:G-flow}
\end{figure}

Let $X$ be a $G$-family of racks and let $(D,\varphi)$ be a $G$-flowed diagram of an oriented spatial surface.
An \textit{$X$-coloring} of $(D,\varphi)$ is a map $C : \mathcal{A}(D) \to X$ satisfying the conditions depicted in Fig.~\ref{fig:G-family of rack coloring} at each crossing and vertex of $D$.
We denote by $\col_X(D,\varphi)$ the set of all $X$-colorings of $(D,\varphi)$.

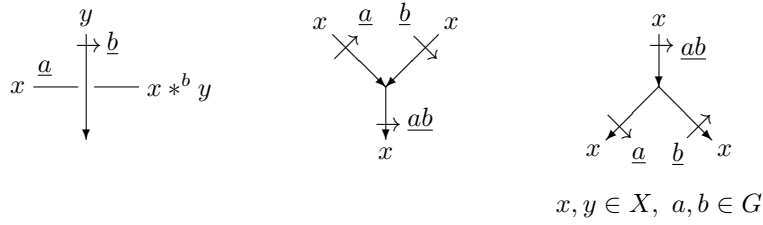
\begin{figure}[htb]
\mbox{}\hfill
\begin{picture}(70,80)
 \put(10,50){\line(1,0){17}}
 \put(50,50){\line(-1,0){17}}
 \put(30,70){\vector(0,-1){40}}
 \put(31,65){\makebox(0,0){$\rightarrow$}}
 \put(40,62){\makebox(0,0)[b]{$\underline{b}$}}
 \put(17,57){\makebox(0,0)[r]{$\underline{a}$}}
  \put(30,73){\makebox(0,0)[b]{$y$}}
 \put(7,50){\makebox(0,0)[r]{$x$}}
 \put(53,50){\makebox(0,0)[l]{$x*^by$}}
\end{picture}
\hfill
\begin{picture}(60,80)
 \put(10,70){\vector(1,-1){20}}
 \put(50,70){\vector(-1,-1){20}}
 \put(30,50){\vector(0,-1){20}}
 \put(15,65){\makebox(0,0){$\nearrow$}}
 \put(45,65){\makebox(0,0){$\searrow$}}
 \put(31,35){\makebox(0,0){$\rightarrow$}}
 \put(25,73){\makebox(0,0)[br]{$\underline{a}$}}
 \put(35,73){\makebox(0,0)[bl]{$\underline{b}$}}
 \put(43,42){\makebox(0,0)[t]{$\underline{ab}$}}
  \put(8,72){\makebox(0,0)[br]{$x$}}
 \put(52,72){\makebox(0,0)[bl]{$x$}}
 \put(30,27){\makebox(0,0)[t]{$x$}}
\end{picture}
\hfill
\begin{picture}(60,80)
 \put(30,50){\vector(1,-1){20}}
 \put(30,50){\vector(-1,-1){20}}
 \put(30,70){\vector(0,-1){20}}
 \put(15,35){\makebox(0,0){$\searrow$}}
 \put(45,35){\makebox(0,0){$\nearrow$}}
 \put(31,65){\makebox(0,0){$\rightarrow$}}
 \put(25,27){\makebox(0,0)[tr]{$\underline{a}$}}
 \put(35,29){\makebox(0,0)[tl]{$\underline{b}$}}
 \put(43,60){\makebox(0,0)[b]{$\underline{ab}$}}
  \put(8,28){\makebox(0,0)[tr]{$x$}}
 \put(52,28){\makebox(0,0)[tl]{$x$}}
 \put(30,73){\makebox(0,0)[b]{$x$}}
  \put(30,5){\makebox(0,0){$x,y \in X,~a,b\in G$}}
\end{picture}
\hfill\mbox{}
\caption{The coloring condition of a $G$-flowed diagram.}
\label{fig:G-family of rack coloring}
\end{figure}

Let $X \times G$ be the associated multiple group rack of a $G$-family of racks $X$.
Let $\pr_G$ and $\pr_X$ be the natural projections from $X \times G$ to $G$ and from $X \times G$ to $X$, respectively.
Let $D$ be a Y-oriented diagram of an oriented spatial surface.
For $\varphi \in \flow (D;G)$, 
we can identify 
an $X$-coloring of $(D,\varphi)$ 
with an $(X\times G)$-coloring $C$ of $D$ 
satisfying $\pr_G \circ C=\varphi$, 
that is, 
for any $C \in \col_{X \times G}(D)$, 
the map $\pr_G \circ C$ corresponds to the $G$-flow of $D$, 
and the map $\pr_X \circ C$ corresponds to the $X$-coloring of $(D,\pr_G \circ C)$.
Under the identification, we have 
\begin{align*}
\col_{X \times G}(D)
&=\bigsqcup_{\varphi \in \flow(D;G)}\{ C \in \col_{X \times G}(D) \,|\, \pr_G \circ C=\varphi \}\\
&=\bigsqcup_{\varphi \in \flow(D;G)}\col_X(D,\varphi).
\end{align*}
Since a $G$-flow $\varphi$ of $F$ depends only on $F$, 
we have the following proposition 
by Theorem~\ref{thm:coloring}.

\begin{proposition}\label{prop:G-family coloring}
Let $X$ be a $G$-family of racks.
Let $D$ and $D'$ be Y-oriented diagrams of an oriented spatial surface $F$ 
and let $\varphi$ be a $G$-flow of $F$.
Then there is a one-to-one correspondence between $\col_X(D,\varphi)$ and $\col_X(D',\varphi)$.
\end{proposition}

By Theorem~\ref{thm:coloring} and Proposition~\ref{prop:G-family coloring}, 
we have the following corollary immediately.

\begin{corollary}\label{cor:multiset inv}
Let $X$ be a $G$-family of racks 
and let $X \times G$ be the associated multiple group rack of $X$.
Let $D$ be a Y-oriented diagram of an oriented spatial surface $F$.
Then the multiset $\{\# \col_X(D,\varphi) \mid \varphi \in \flow(D;G) \}$ 
is an invariant of $F$.
\end{corollary}

\section{Coloring examples}

In this section, 
we give some examples of colorings for oriented spatial surfaces 
by using multiple group racks.
In Example~\ref{ex:coloring1}, 
we distinguish the two oriented spatial surfaces 
which we could not in Section~2.
In Example~\ref{ex:coloring2}, 
we distinguish two oriented Seifert surfaces of an oriented knot, 
which were distinguished in~\cite{Alford70}.
In Example~\ref{ex:coloring3}, 
we distinguish two oriented spatial closed surfaces 
which do not bound handlebodies in $S^3$.
In this section, 
we represent the multiplicity of an element of a multiset 
by a subscript with an underline.
For example, $\{a_{\underline{1}}, b_{\underline{2}}, c_{\underline{3}}\}$ represents the multiset $\{a, b, b, c, c, c\}$.

\begin{example}\label{ex:coloring1}
Let $F_1$ and $F_2$ be the oriented spatial surfaces illustrated in Fig.~\ref{fig:not invertible2}, 
where $F_2$ is the reverse of $F_1$.
We remind that 
they can not be distinguished by using the methods introduced in Section~2.
Let $(D_1,\varphi_1(a,b,c))$ and $(D_2,\varphi_2(a,b,c))$ be their $\Z_2$-flowed diagrams 
illustrated in Fig.~\ref{fig:not invertible2 diagram} for $a,b,c \in \Z_2$, respectively.
We note that $\flow(D_i;\Z_2)=\{ \varphi_i(a,b,c) \mid (a,b,c) \in \Z_2^3 \}$ for each $i=1,2$.
Let $R_3^3$ be the rack we investigated in Proposition~\ref{prop:rack type}.
Then $R_3^3$ is a $\Z_2$-family of racks by Proposition~\ref{prop:G-family of racks}.
Any $R_3^3$-coloring of $(D_1,\varphi_1(a,b,c))$ is given in the left of Fig.~\ref{fig:not invertible2 diagram} 
with the following relations:
\begin{align*}
&x*^ax=x,\\
&x*^c(x*^ax)=x*^ax,\\
&(x*^ay)*^ax=y,\\
&y*^a(x*^ay)=x
\end{align*}
for $x,y \in R_3^3$, 
which are required at the crossings $c_1,c_2,c_3$ and $c_4$, respectively.
Then we obtain that
\[ \{\# \col_{R_3^3}(D_1,\varphi) \mid \varphi \in \flow(D_1;\Z_2) \}
=\{3_{\underline{2}}, 9_{\underline{4}}, 27_{\underline{2}}\}.\]
On the other hand, 
any $R_3^3$-coloring of $(D_2,\varphi_2(a,b,c))$ is given in the right of Fig.~\ref{fig:not invertible2 diagram} 
with the following relations:
\begin{align*}
&x*^cx=x,\\
&(x*^ay)*^ax=y,\\
&y*^a(x*^ay)=(x*^cx)*^b(x*^cx)
\end{align*}
for $x,y \in R_3^3$, 
which are required at the crossings $c_5,c_6$ and $c_7$, respectively.
Then we obtain that
\[ \{\# \col_{R_3^3}(D_2,\varphi) \mid \varphi \in \flow(D_2;\Z_2) \}
=\{3_{\underline{3}}, 9_{\underline{3}}, 27_{\underline{1}}, 81_{\underline{1}}\}.\]
Consequently, 
$F_1$ and $F_2$ are not equivalent by Corollary~\ref{cor:multiset inv}, 
that is, they are not reversible.
We remark that $F_1$ and $F_2$ can not be distinguished by using $R_3$.
\end{example}

	\begin{figure}[htpb]
	\begin{center}
	\includegraphics{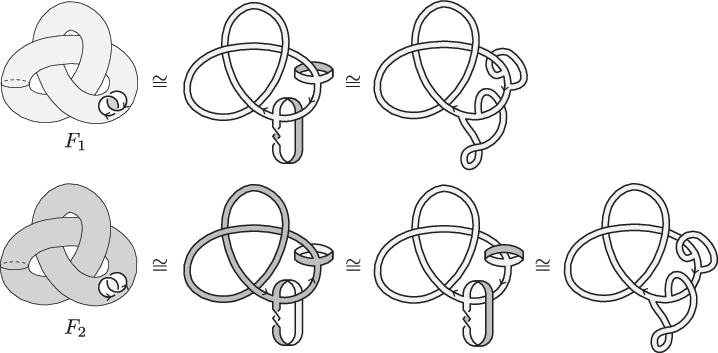}
	\caption{An oriented spatial surface $F_1$ and its reverse $F_2$.}
	\label{fig:not invertible2}
	\end{center}
	\end{figure}
	
	\begin{figure}[htpb]
	\begin{center}
	\includegraphics{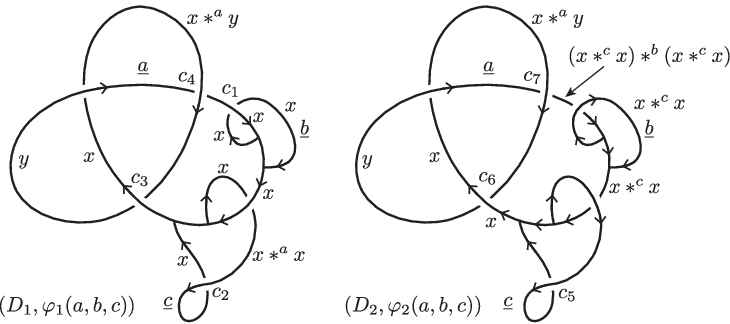}
	\caption{$\Z_2$-flowed diagrams $(D_1,\varphi_1(a,b,c))$ and $(D_2,\varphi_2(a,b,c))$.}
	\label{fig:not invertible2 diagram}
	\end{center}
	\end{figure}

\begin{remark}
In Example~\ref{ex:coloring1}, 
we can also distinguish $F_1$ and $F_2$ by using a $(t,s)$-rack.
Put $s:=2t^2+2 \in \Z_3[t^{\pm 1}]$, 
and let $X$ be the $(t,s)$-rack $\Z_3[t^{\pm 1}]/(t^4+2 t^3+2t+2)$ 
with $x*y=tx+sy$ for any $x,y \in X$.
Since $\type X=8$, 
then $X$ is a $\Z_8$-family of racks by Proposition~\ref{prop:G-family of racks}.
With the $\Z_8$-family of racks, 
we obtain that
\[ \{\# \col_X(D_1,\varphi) \mid \varphi \in \flow(D_1;\Z_8) \}
=\{9_{\underline{432}}, 81_{\underline{64}}, 729_{\underline{16}}\}\]
and
\[ \{\# \col_X(D_2,\varphi) \mid \varphi \in \flow(D_2;\Z_8) \}
=\{9_{\underline{480}}, 81_{\underline{32}}\}.\]
\end{remark}

\begin{example}\label{ex:coloring2}
Let $F_1$ and $F_2$ be the spatial surfaces illustrated in Fig.~\ref{fig:Seifert sfc}.
We note that $F_1$ and $F_2$ are Seifert surfaces of the oriented knot $K$ 
illustrated in the figure, 
which were distinguished in~\cite{Alford70}.
Let $(D_1,\varphi_1(a,b))$ and $(D_2,\varphi_2(a,b))$ be their $\Z_2$-flowed diagrams 
illustrated in Fig.~\ref{fig:Seifert sfc diagram} for $a,b \in \Z_2$, respectively.
We note that $\flow(D_i;\Z_2)=\{ \varphi_i(a,b) \mid (a,b) \in \Z_2^2 \}$ for each $i=1,2$.
The dihedral quandle $R_3$ is a $\Z_2$-family of quandles by Proposition~\ref{prop:G-family of racks}.
Any $R_3$-coloring of $(D_1,\varphi_1(a,b))$ is given in the left of Fig.~\ref{fig:Seifert sfc diagram} 
with the following relations:
\begin{align*}
&(w*^aB)*^aA=x,\\
&(z*^aB)*^aA=y,\\
&(A*^aB)*^aA=A,\\
&((A*^{-a}(B*^bA))*^ay)*^ax=z,\\
&((B*^bA)*^ay)*^ax=w
\end{align*}
for $x,y,z,w \in R_3$, 
where $A=(x*^az)*^aw$ and $B=(y*^az)*^aw$.
These relations are required at the crossings $c_1,c_2,c_3,c_4$ and $c_5$, respectively.
Then we obtain that
\[ \{\# \col_{R_3}(D_1,\varphi) \mid \varphi \in \flow(D_1;\Z_2) \}
=\{3_{\underline{4}}\}.\]
On the other hand, 
any $R_3$-coloring of $(D_2,\varphi_2(a,b))$ is given in the right of Fig.~\ref{fig:Seifert sfc diagram} 
with the following relations:
\begin{align*}
&y*^a(x*^ay)=x,\\
&(x*^ay)*^ax=y
\end{align*}
for $x,y \in R_3$, 
which are required at the crossings $c_6$ and $c_7$, respectively.
Then we obtain that
\[ \{\# \col_{R_3}(D_2,\varphi) \mid \varphi \in \flow(D_2;\Z_2) \}
=\{3_{\underline{2}}, 9_{\underline{2}}\}.\]
Consequently, 
$F_1$ and $F_2$ are not equivalent by Corollary~\ref{cor:multiset inv}.
\end{example}

	\begin{figure}[htpb]
	\begin{center}
	\includegraphics{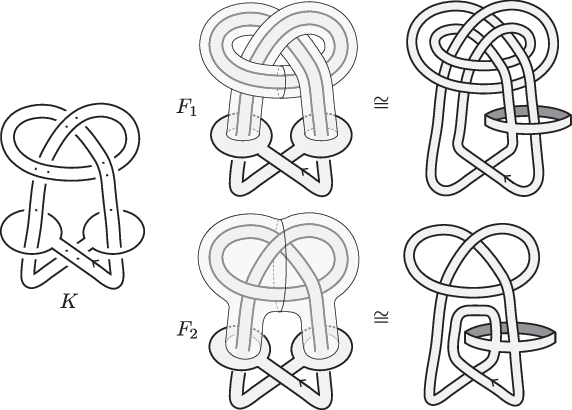}
	\caption{Two Seifert surfaces $F_1$ and $F_2$ of the oriented knot $K$.}
	\label{fig:Seifert sfc}
	\end{center}
	\end{figure}
	
	\begin{figure}[htpb]
	\begin{center}
	\includegraphics{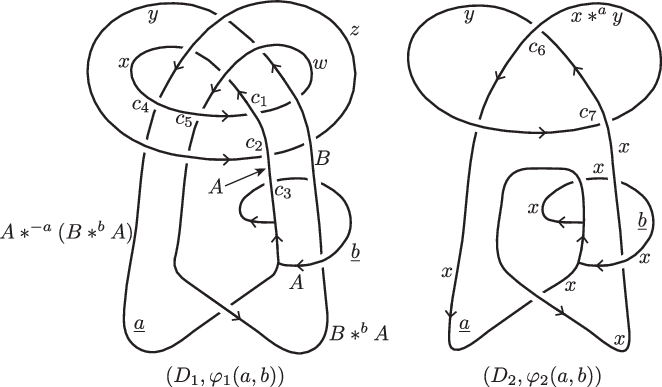}
	\caption{$\Z_2$-flowed diagrams $(D_1,\varphi_1(a,b))$ and $(D_2,\varphi_2(a,b))$.}
	\label{fig:Seifert sfc diagram}
	\end{center}
	\end{figure}

\begin{example}\label{ex:coloring3}
Let $\widetilde{F_1}$ and $\widetilde{F_2}$ be the oriented spatial closed surfaces 
illustrated in Fig.~\ref{fig:Homma sfc}, 
where we note that they do not bound handlebodies in $S^3$.
Let $F_1$ and $F_2$ be the oriented spatial surfaces 
obtained by removing a disk from $\widetilde{F_1}$ and $\widetilde{F_2}$ respectively 
as shown in the figure.
By~\cite{Matsuzaki19}, 
$\widetilde{F_1}$ and $\widetilde{F_2}$ are equivalent 
if and only if 
$F_1$ and $F_2$ are equivalent.
Let $(D_1,\varphi_1(a,b,c,d))$ and $(D_2,\varphi_2(a,b,c,d))$ be the $\Z_2$-flowed diagrams 
of $F_1$ and $F_2$ 
illustrated in Fig.~\ref{fig:Homma sfc diagram} for $a,b,c,d \in \Z_2$, respectively.
We note that $\flow(D_i;\Z_2)=\{ \varphi_i(a,b,c,d) \,|\, (a,b,c,d) \in \Z_2^4 \}$ for each $i=1,2$.
The dihedral quandle $R_3$ is a $\Z_2$-family of quandles by Proposition~\ref{prop:G-family of racks}.
Any $R_3$-coloring of $(D_1,\varphi_1(a,b,c,d))$ is given 
in the left of Fig.~\ref{fig:Homma sfc diagram} 
with the following relations:
\begin{align*}
&w*^a(y*^aw)=y,\\
&z*^a(y*^aw)=(y*^dx)*^{-b}y,\\
&(y*^aw)*^ay=w,\\
&(x*^aw)*^ay=z
\end{align*}
for $x,y,z,w \in R_3$, 
which are required at the crossings $c_1,c_2,c_3$ and $c_4$, respectively.
Then we obtain that
\[ \{\# \col_{R_3}(D_1,\varphi) \mid \varphi \in \flow(D_1;\Z_2) \}
=\{3_{\underline{6}}, 9_{\underline{8}}, 27_{\underline{2}}\}.\]
On the other hand, 
any $R_3$-coloring of $(D_2,\varphi_2(a,b,c,d))$ is given 
in the right of Fig.~\ref{fig:Homma sfc diagram} 
with the following relations:
\begin{align*}
&w*^{-a}(w*^{-a}(y*^aw))=y,\\
&z*^{-a}(w*^{-a}(y*^aw))=(y*^dx)*^{-b}y,\\
&(z*^{-a}(y*^aw))*^ay=x*^aw,\\
&(w*^{-a}(y*^aw))*^ay=y*^aw
\end{align*}
for $x,y,z,w \in R_3$, 
which are required at the crossings $c_5,c_6,c_7$ and $c_8$, respectively.
Then we obtain that
\[\{\# \col_{R_3}(D_2,\varphi) \mid \varphi \in \flow(D_2;\Z_2) \}
=\{3_{\underline{12}}, 9_{\underline{4}}\}.\]
Consequently, 
$F_1$ and $F_2$ are not equivalent by Corollary~\ref{cor:multiset inv}, 
that is, $\widetilde{F_1}$ and $\widetilde{F_2}$ are not equivalent.
\end{example}

	\begin{figure}[htpb]
	\begin{center}
	\includegraphics{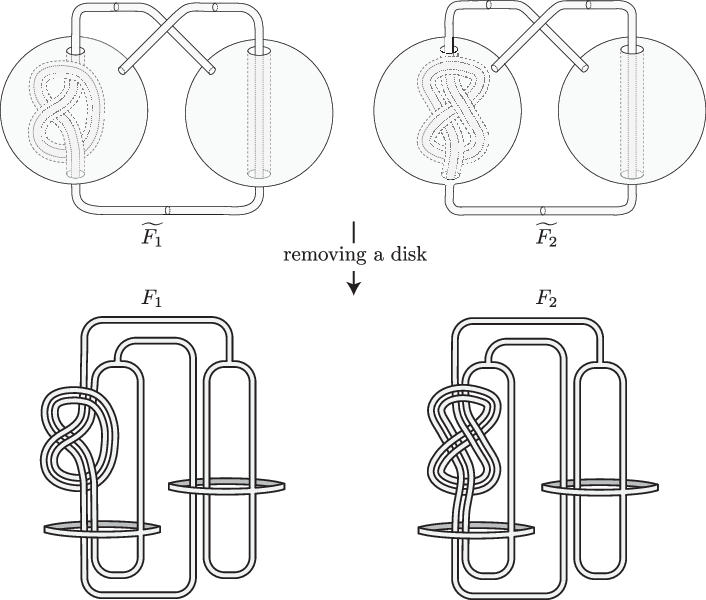}
	\caption{Two oriented spatial closed surfaces $\widetilde{F_1}$ and $\widetilde{F_2}$.}
	\label{fig:Homma sfc}
	\end{center}		
	\end{figure}
	
	\begin{figure}[htpb]
	\begin{center}
	\includegraphics{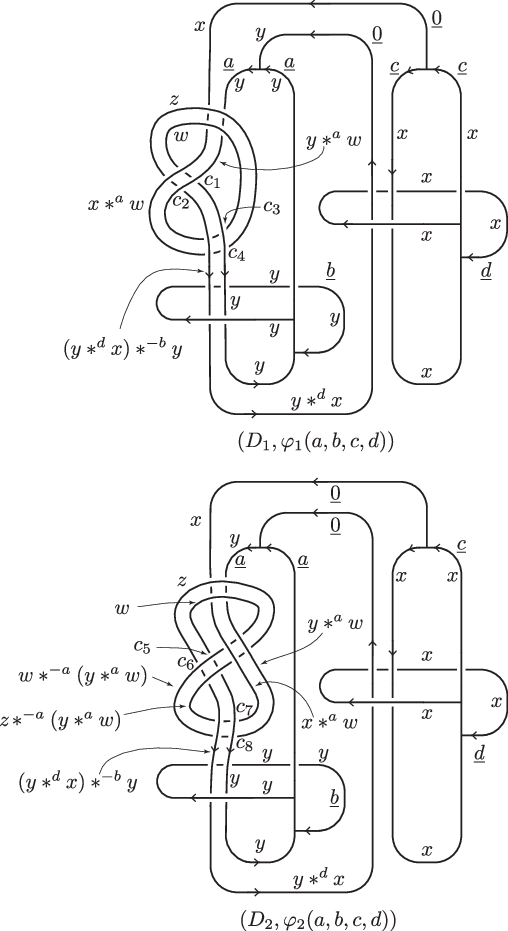}
	\caption{$\Z_2$-flowed diagrams $(D_1,\varphi_1(a,b,c,d))$ and $(D_2,\varphi_2(a,b,c,d))$.}
	\label{fig:Homma sfc diagram}
	\end{center}		
	\end{figure}

\begin{remark}
Example~\ref{ex:coloring1} gives a counter example 
for the converse of Remark~\ref{rem:elementary method} (1) and (2).
In Examples~\ref{ex:coloring2} and~\ref{ex:coloring3}, 
we showed that 
the oriented spatial surfaces $F_1$ and $F_2$ are not equivalent 
by using multiple conjugation quandle colorings.
This means that 
the regular neighborhoods $N(F_1)$ and $N(F_2)$ are not equivalent 
as handlebody-knots~\cite{Ishii15-1}.
\end{remark}

\section*{Acknowledgements}
The first author was supported by JSPS KAKENHI Grant Number 18K03292.
The third author was supported by JSPS KAKENHI Grant Number 18J10105.


\end{document}